\documentclass[reqno]{amsart}
\usepackage{fullpage}
\usepackage{array}
\usepackage{mathrsfs}
\usepackage{multirow}
\usepackage{hyperref}
\usepackage{graphicx,color}
\usepackage{float}
\usepackage{ytableau}

\newcommand{\varied}

\newcommand{\red}{\color{red}}

\usepackage{comment}

\newcommand{\Id}{\textnormal{Id}}

\usepackage{longtable}

\newtheorem{theorem}{Theorem}[section]
\newtheorem{example}[theorem]{Example}

\newtheorem{lemma}[theorem]{Lemma}
\newtheorem{corollary}[theorem]{Corollary}

\newtheorem{proposition}[theorem]{Proposition}

\newcolumntype{C}[1]{>{\centering\let\newline\\\arraybackslash\hspace{0pt}}m{#1}}

\usepackage[utf8]{inputenc}
\usepackage{amsfonts}
\usepackage{amssymb}
\usepackage{verbatim}
\usepackage{booktabs}
\usepackage{comment}

\numberwithin{equation}{section}

\begin{document}
\title{Algebras with Bounded Multiplicities of the \\ Central Cocharacters}

\author{Wesley Q. Cota$^{1,a,*}$, Luiz H. S. Matos$^{2,b}$ and Breno O. Souza$^{3,b}$}

\thanks{%
  \rightskip=0pt plus 1fil\relax
  \textit{E-mail addresses:}
  wesleyqc@ime.usp.br~(Cota),
  luizmatos@ufmg.br~(Matos), breno55@ufmg.br~(Souza).
}

\thanks{\footnotesize $^{1}$ Supported by FAPESP, grant nos.~2024/14914-9, 2025/05699-0 and 2026/06588-0}

\thanks{\footnotesize $^{2}$ Partially supported by FAPEMIG}

\thanks{\footnotesize $^3$ Partially supported by CAPES}

\thanks{\footnotesize $^{*}$ Corresponding author}



\dedicatory{$^{a}$ IME, USP, Rua do Matão 1010, 05508-090, São Paulo, Brazil \\ $^b$ ICEx, UFMG, Avenida Antonio Carlos 6627, 31123-970, Belo Horizonte, Brazil}

\begin{abstract} Let $A$ be an associative PI-algebra over a field of characteristic zero. The central polynomials of $A$ give rise to a sequence $\{\chi_n^z(A)\}_{n\geq 1}$ of $S_n$-characters, called its central cocharacters. A natural problem is to understand which algebraic properties of $A$ are imposed by bounds on the multiplicities occurring in the decompositions of these cocharacters into irreducible characters. We prove that these multiplicities are bounded by one if and only if there exist scalars $(\alpha,\beta)\neq(0,0)$ such that $\alpha[x_1,x_2]x_1+\beta x_1[x_1,x_2]$ takes values in the center of $A$. This provides a central analogue of the Ananin-Kemer theorem for ordinary cocharacters. We also characterize the algebras whose central cocharacter multiplicities are uniformly bounded.

\end{abstract}

\maketitle

\section{Introduction}

The symmetric group acts on multilinear polynomials by permuting their variables and this action induces an action on the quotients modulo polynomial identities or central polynomials of an associative algebra. This construction connects the study of these polynomials
with the representation theory of symmetric groups. Whereas
codimensions record only the dimensions of the resulting quotient
spaces, cocharacters describe their irreducible constituents and
the multiplicities with which they occur. The cocharacter sequence
thus provides a finer invariant, capturing the representation-theoretic
structure of the multilinear quotients. Bounds on multiplicities
impose restrictions on this structure that dimensions alone do not
determine. Of particular interest is the case in which each irreducible constituent appears at most once in each degree. A natural question is whether a polynomial relation of low
degree can force this condition in every degree.

We work over a field $F$ of characteristic zero and consider
associative $F$-algebras, not necessarily with an identity element.
For an algebra $A$, we denote by $\Id(A)$ the set of its polynomial
identities and by $\Id^z(A)$ the set of polynomials whose evaluations
always lie in the center $Z(A)$ of $A$. In particular,
$\Id(A)\subseteq\Id^z(A)$. In characteristic zero, multilinearization
shows that these spaces are generated by their multilinear elements
as a $T$-ideal and a $T$-space, respectively. Their multilinear
components therefore determine all polynomial identities and central
polynomials of $A$, making the associated representations fundamental
to the study of both.

Let $P_n$ be the space of multilinear polynomials in
$x_1,\ldots,x_n$. The corresponding quotient $S_n$-modules are
\[
P_n(A)=\frac{P_n}{P_n\cap\Id(A)}
\qquad\text{and}\qquad
P_n^z(A)=\frac{P_n}{P_n\cap\Id^z(A)}.
\]
Their characters, called the ordinary and central cocharacters,
respectively, decompose as
\[
\chi_n(A)=\sum_{\lambda\vdash n}m_\lambda(A)\chi_\lambda,
\qquad
\chi_n^z(A)=\sum_{\lambda\vdash n}m_\lambda^z(A)\chi_\lambda,
\]
where $\chi_\lambda$ is the irreducible $S_n$-character indexed by
the partition $\lambda$ of $n$ and $m_\lambda(A)$ and $m_\lambda^z(A)$
are the ordinary and central multiplicities. We use analogous
notation for a variety $\mathcal V$. A family of multiplicities
is uniformly bounded if it admits a bound independent of both
$n$ and $\lambda\vdash n$.

A fundamental result in this setting was established by Ananin
and Kemer \cite{AnaninKemer1976} in 1976. They proved that all
ordinary cocharacter multiplicities of a variety $\mathcal V$
are bounded by one if and only if
\[
\alpha[x_1,x_2]x_1+\beta x_1[x_1,x_2]\in\Id(\mathcal V)
\]
for some $\alpha,\beta\in F$, not both zero. Their theorem also
identifies this condition with the distributivity of the lattice
of subvarieties of $\mathcal V$. It thus connects this identity with both a representation-theoretic restriction in every degree
and a structural property of the subvariety lattice.

For arbitrary uniform bounds, Mishchenko, Regev and Zaicev
\cite{MishchenkoRegevZaicev1999} proved that the ordinary
multiplicities of $\mathcal V$ are bounded if and only if
$UT_2\notin\mathcal V$, where $UT_2$ denotes the algebra of
$2\times2$ upper triangular matrices over $F$. More recently,
Giambruno, Quintino and Vieira
\cite{GiambrunoQuintinoVieira2026} classified, up to PI-equivalence,
the algebras whose ordinary cocharacter multiplicities are bounded
by one.

The passage to central polynomials introduces a substantial
obstacle. Whereas $\Id(A)$ is a $T$-ideal, $\Id^z(A)$ is generally
only a $T$-space: it is stable under endomorphisms of the free
algebra, but need not be stable under multiplication. In particular,
a central polynomial $f$ need not remain central after multiplication
by an additional variable. Arguments that derive higher-degree
consequences by multiplying identities therefore do not carry over
directly. Extending the Ananin--Kemer theorem to central cocharacters
requires a different way to obtain sufficient control in all degrees
from a cubic centrality condition.

We first characterize varieties with uniformly bounded central
multiplicities. We prove that these multiplicities are uniformly bounded
if and only if $UT_2\notin\mathcal V$. Combined with the theorem
of Mishchenko, Regev and Zaicev \cite{MishchenkoRegevZaicev1999}, this shows that central
multiplicities are uniformly bounded exactly when ordinary
multiplicities are uniformly bounded. We also obtain equivalent
characterizations in terms of the partitions occurring in the
central cocharacters and the existence of specific central
polynomials.

Our main result establishes the analogue of the Ananin--Kemer theorem in the central polynomials setting.

\begin{theorem}
Let $\mathcal V$ be a variety. Then
$m_\lambda^z(\mathcal V)\leq1$ for every $n\geq1$ and every
$\lambda\vdash n$ if and only if
$\alpha[x_1,x_2]x_1+\beta x_1[x_1,x_2]\in\Id^z(\mathcal V)$ for some $\alpha,\beta\in F$, not both zero.
\end{theorem}

This condition is equivalent to
$m_{(2,1)}^z(\mathcal V)\leq1$. Thus a restriction in degree three already forces the multiplicity-one condition in every degree. The proof treats separately the possible relations between $\alpha$ and $\beta$. The infinite-dimensional Grassmann algebra, either alone or in suitable direct sums, provides the comparison cocharacters needed to establish the multiplicity bounds.

The paper is organized as follows. Section~2 recalls the background
on ordinary and central codimensions and cocharacters. Section~3
gives equivalent characterizations of varieties with uniformly
bounded central multiplicities. Section~4 studies the algebras used
in the proofs and determines their relevant central $T$-spaces
and cocharacters. Section~5 proves the main theorem on central
multiplicities bounded by one.

\section{Preliminaries}

Throughout the paper, $F$ denotes a field of characteristic zero and all algebras are associative $F$-algebras, not necessarily with an identity element. Let $X=\{x_1,x_2,\ldots\}$ be a countable
set of variables. We denote by $F\langle X\rangle$ the free associative algebra without identity over $F$ generated by $X$.

A polynomial
$f=f(x_1,\ldots,x_n)\in F\langle X\rangle$ is called a \emph{central polynomial} of $A$ if $f(a_1,\ldots,a_n)\in Z(A)$ for all $a_1,\ldots,a_n\in A$, where $Z(A)$ is the center of $A$. If $f(a_1,\ldots,a_n)=0$ for all
$a_1,\ldots,a_n\in A$, then $f$ is a polynomial identity of $A$.
A central polynomial is called a \emph{proper central polynomial} if it is
not a polynomial identity. An algebra satisfying a nonzero polynomial
identity is called a \emph{PI-algebra}.

The sets of central polynomials and polynomial identities of $A$ are
denoted, respectively, by $\Id^z(A)$ and $\Id(A)$.
The space $\Id^z(A)$ is a $T$-space, that is, a vector subspace of
$F\langle X\rangle$ invariant under all endomorphisms of the free
algebra. Moreover, $\Id(A)$ is a $T$-ideal, meaning that it is both
a $T$-space and a two-sided ideal of $F\langle X\rangle$.
For a set $S\subseteq F\langle X\rangle$, we denote by
$\langle S\rangle_T$ the $T$-ideal generated by $S$ and $\langle S\rangle^T$ the $T$-space generated by $S$. A polynomial $f$
is called a consequence of $S$ in the $T$-ideal sense
if
$f\in\langle S\rangle_T$ and in the $T$-space sense if
$f\in\langle S\rangle^T$.

For every $f=f(x_1,\ldots,x_n)\in F\langle X\rangle$, we have $f\in\operatorname{Id}^z(A)$ if and only if $[f,x_{n+1}]\in\operatorname{Id}(A).$ In particular, the $T$-space of central polynomials is determined by the $T$-ideal of polynomial identities.

Kemer \cite{Kemer1988} proved that every $T$-ideal over a field of
characteristic zero is finitely generated as a $T$-ideal. Later,
Shchigolev \cite{Shchigolev2001} established the corresponding finite
basis property for $T$-spaces. Furthermore, since $F$ has characteristic
zero, the multilinearization process shows that both $\Id(A)$
and $\Id^z(A)$ are generated by their multilinear elements.

For each $n\geq1$, let $P_n$ be the space of multilinear polynomials
in the variables $x_1,\ldots,x_n$, namely,
\[
P_n=\operatorname{span}_F
\{x_{\sigma(1)}x_{\sigma(2)}\cdots x_{\sigma(n)}
\mid \sigma\in S_n\}.
\]

We write $[x_1,x_2]=x_1x_2-x_2x_1$ and define, for $k\geq 3$, $[x_1,\ldots,x_k]=[[x_1,\ldots,x_{k-1}],x_k].$ Moreover, the standard polynomial of degree $n$ is defined by $$St_n=St_n(x_1, \ldots, x_n)=\sum_{\sigma \in S_n} (\mbox{sgn }\sigma )x_{\sigma(1)}\ldots x_{\sigma(n)}.$$

To investigate the asymptotic behavior of the polynomial identities,
central polynomials and proper central polynomials of $A$,
Regev \cite{Regev1972,Regev2016} introduced the following quotient spaces:
\[
P_n(A)=\frac{P_n}{P_n\cap\Id(A)},\qquad
P_n^z(A)=\frac{P_n}{P_n\cap\Id^z(A)},\qquad
P_n^\delta(A)=
\frac{P_n\cap\Id^z(A)}{P_n\cap\Id(A)}.
\]
Their dimensions are denoted by $c_n(A)$, $c_n^z(A)$,
and $c_n^\delta(A)$ and are called the \emph{$n$th codimension},
\emph{central codimension} and \emph{proper central codimension} of $A$,
respectively. In particular,
\[
c_n(A)=c_n^z(A)+c_n^\delta(A)
\qquad\text{for every }n\geq1.
\]

We denote by $\mathcal V=\operatorname{var}(A)$ the \emph{variety} generated
by $A$, that is, the class of all $F$-algebras satisfying every
polynomial identity of $A$. We have $\Id(\mathcal V)=\Id(A)$ and
$\Id^z(\mathcal V)=\Id^z(A)$ and define the codimensions of
$\mathcal V$ by $c_n(\mathcal V)=c_n(A)$,
$c_n^z(\mathcal V)=c_n^z(A)$ and
$c_n^\delta(\mathcal V)=c_n^\delta(A)$.

The symmetric group $S_n$ acts naturally on $P_n$ by permuting the
variables. Under this action, $P_n$ is isomorphic to the left
$S_n$-module $FS_n$ via the linear map
\begin{equation}\label{eq:regular-module}
FS_n\rightarrow P_n,
\qquad
\sigma\mapsto x_{\sigma(1)}\cdots x_{\sigma(n)}.
\end{equation}
Since $\Id(A)$ and $\Id^z(A)$ are $T$-spaces, their intersections with
$P_n$ are invariant under this action. Consequently, the quotients
$P_n(A)$ and $P_n^z(A)$ inherit structures of left $S_n$-modules.
Their characters are denoted by $\chi_n(A)$ and $\chi_n^z(A)$ and
are called the \emph{$n$th cocharacter} and \emph{central cocharacter} of $A$,
respectively.
Since $F$ has characteristic zero,  by complete reducibility, we can write
\begin{equation}\label{character}
\chi_n(A)=\sum_{\lambda\vdash n}m_\lambda(A)\chi_\lambda
\qquad\text{and}\qquad
\chi_n^z(A)=\sum_{\lambda\vdash n}m_\lambda^z(A)\chi_\lambda,
\end{equation}
where $\chi_\lambda$ is the irreducible $S_n$-character associated to the partition $\lambda$ of $n$ and $m_\lambda(A)$
and $m_\lambda^z(A)$ are the corresponding multiplicities.

Under the identification in~\eqref{eq:regular-module}, minimal left
ideals of $FS_n$ correspond to irreducible $S_n$-submodules of $P_n$.
We recall their construction using essential idempotents. For a Young
tableau $T_\lambda$ of shape $\lambda\vdash n$, the essential
idempotent $e_{T_\lambda}\in FS_n$ is defined by
\begin{equation}\label{eq:idempotent}
e_{T_\lambda}
=
\sum_{\sigma\in R_{T_\lambda}}
\sum_{\tau\in C_{T_\lambda}}
(\operatorname{sgn}\tau)\sigma\tau,
\end{equation}
where $R_{T_\lambda}$ and $C_{T_\lambda}$ denote the row and column stabilizers of $T_\lambda$, respectively.

\begin{lemma} \cite[Lemma 2.4.1]{GZ}\label{lem:idempotent_generation}
Let $M$ be an irreducible left $S_n$-module with character
$\chi_\lambda$ and let $0\neq f\in M$. Then there exists a Young
tableau $T_\lambda$ of shape $\lambda$ such that
$e_{T_\lambda}f\neq0$. In particular,
$M=FS_n(e_{T_\lambda}f)$.
\end{lemma}

Let $S_\lambda$ denote the irreducible $S_n$-module corresponding to a partition $\lambda\vdash n$. We write $d_\lambda=\chi_\lambda(1)=\dim_F S_\lambda$.  For every finite-dimensional $FS_n$-module $M$ and every Young tableau $T_\lambda$ of shape $\lambda$, the multiplicity
of $S_\lambda$ in $M$ is
\[
m_\lambda(M)=\dim_F e_{T_\lambda}M.
\] Since $P_n$ is isomorphic to $FS_n$ as an $FS_n$-module, the decompositions in the first three degrees are
\[
P_1\cong S_{(1)},\quad
P_2\cong S_{(2)}\oplus S_{(1,1)}
\quad\text{and}\quad
P_3\cong S_{(3)}\oplus 2S_{(2,1)}\oplus S_{(1,1,1)}.
\]
Thus, $P_1$ and $P_2$, and hence their quotients, have multiplicities bounded by one. In degree three, the only multiplicity that can
exceed one is that corresponding to $\lambda=(2,1)$.

For this partition, consider the essential idempotent
$e_{T_\lambda}=(1+(13))(1-(12))$. The space $e_{T_\lambda}P_3$ has a basis consisting of the polynomials
\[
f=[x_1,x_2]x_3+[x_3,x_2]x_1,
\qquad
g=x_1[x_3,x_2]+x_3[x_1,x_2].
\]

Setting $x_3=x_1$ yields $f(x_1,x_2,x_1)=2\tilde{f}$ and
$g(x_1,x_2,x_1)=2\tilde{g}$, where
\begin{equation*}
\tilde{f}=[x_1,x_2]x_1
\quad \mbox{ and }\quad 
\tilde{g}=x_1[x_1,x_2].
\end{equation*}
Conversely, $f$ and $g$ are the linearizations of $\tilde{f}$
and $\tilde{g}$ with respect to $x_1$.

The multiplicity $m_{(2,1)}(A)$ is the dimension of the span of
the images of $f$ and $g$ in $P_3(A)$. Hence $m_{(2,1)}(A)\leq1$
if and only if $\alpha f+\beta g\in\Id(A)$ for some
$\alpha,\beta\in F$ not both zero. By specialization and
linearization, this is equivalent to
$\alpha\tilde{f}+\beta\tilde{g}\in\Id(A)$.
Such an identity is therefore necessary for all ordinary cocharacter multiplicities to be bounded by one. Ananin and Kemer \cite{AnaninKemer1976} proved that this condition is also sufficient to control the multiplicities in every degree.

\begin{theorem}[{\cite{AnaninKemer1976}}]
\label{thm:ananin-kemer}
Let $\mathcal V$ be a variety generated by a PI-algebra.
Then $m_\lambda(\mathcal V)\leq1$ for every $n\geq1$ and every
partition $\lambda\vdash n$ if and only if there exist
$\alpha,\beta\in F$, not both zero, such that
\[
\alpha[x_1,x_2]x_1+\beta x_1[x_1,x_2]\in\Id(\mathcal V).
\]
\end{theorem}

Since $\Id^z(A)$ is a $T$-space, the same argument applies to
$P_3^z(A)$. Thus $m_{(2,1)}^z(A)\leq1$ if and only if there
exist $\alpha,\beta\in F$, not both zero, such that
\begin{equation}\label{eq:degree-three-central}
\alpha[x_1,x_2]x_1+\beta x_1[x_1,x_2]\in\Id^z(A).
\end{equation}
Therefore, this is a necessary condition for all central cocharacter multiplicities to be bounded by one. One of the main results of this paper shows that this condition is also sufficient.

\section{When the central multiplicities are bounded by a constant}

The boundedness of multiplicities in ordinary cocharacters was characterized by Mishchenko, Regev and Zaicev \cite[Theorem~1]{MishchenkoRegevZaicev1999} in terms of the exclusion of $UT_2$ from the variety. Since
$$
m_\lambda^z(\mathcal V)\leq m_\lambda(\mathcal V),
$$
bounded ordinary multiplicities always imply bounded central multiplicities. The converse is not immediate, since passing from the ordinary to the central cocharacter may decrease multiplicities. The purpose of this section is to prove that, nevertheless, boundedness of the central multiplicities is equivalent to boundedness of the ordinary ones and to obtain the corresponding central analogues of the characterizations in \cite[Theorem~1]{MishchenkoRegevZaicev1999}. The same algebra is naturally relevant in the central setting. Indeed,
$UT_2$ has no proper central polynomial, see \cite[Lemma 1]{GiambrunoZaicev2018}. Consequently,
$$
\chi_n^z(UT_2)=\chi_n(UT_2)
$$
for every $n\geq1$. We recall below the precise information on its
cocharacters that will be used throughout this section.

\begin{lemma} \cite{Malcev1971,MishchenkoRegevZaicev1999}\label{lem:UT2-central}
For the algebra $UT_2$ we have
$
\operatorname{Id}^z(UT_2)
=
\operatorname{Id}(UT_2)
=
\left\langle [x_1,x_2][x_3,x_4]\right\rangle_T
$ and
\begin{align*}
\chi_n^z(UT_2)=
\chi_n(UT_2)
=
\chi_{(n)}+\sum_{i=1}^{\lfloor n/2\rfloor}
(n-2i+1)\chi_{(n-i,i)}+
\sum_{i=1}^{\lfloor (n-1)/2\rfloor}
(n-2i)\chi_{(n-i-1,i,1)}.
\end{align*}
\end{lemma}

In particular, the multiplicities of the central cocharacters of $UT_2$ are not bounded by a constant. We now combine Lemma~\ref{lem:UT2-central} with the theorem of Mishchenko, Regev and Zaicev to obtain the corresponding characterization
for central cocharacters.


\begin{theorem} 
Let $\mathcal{V}$ be a variety of algebras and let
$\chi_n^z(\mathcal{V})=\sum_{\lambda \vdash n} m_\lambda ^z\chi_\lambda$ be its $n$th central cocharacter. Then the following conditions are equivalent: 
\begin{enumerate}
    \item $m_\lambda ^z\leq k$ for some constant $k$, for all $n$ and for every $\lambda \vdash n$.
\item $U T_2 \notin \mathcal{V}$.
\item There exists a constant $M$ such that for all $n$ and for every $\lambda \vdash n$, the inequality
$$
n-\lambda_1-\lambda_1^{\prime} \leq M
$$
holds whenever $m_\lambda^z \neq 0$, where $\lambda^{\prime}$ is the conjugate partition of $\lambda$.
\item $\mathcal{V}$ admits a nonzero central polynomial of the form
$$
g(x,y)=\sum_{i=0}^m \alpha_i y^i x y^{m-i}, \quad \text{for some  } m\geq0.
$$
\end{enumerate}
\end{theorem}

\begin{proof}
Suppose first that $UT_2\notin\mathcal V$. By \cite[Theorem~1]{MishchenkoRegevZaicev1999},
the ordinary cocharacter multiplicities of $\mathcal V$ are uniformly bounded, the partitions with nonzero ordinary multiplicity satisfy the bound in~(3) and $\mathcal V$ satisfies a nonzero polynomial identity of the form given in~(4). Since
$m_\lambda^z(\mathcal V)\le m_\lambda(\mathcal V)$ and $\operatorname{Id}(\mathcal V)\subseteq
\operatorname{Id}^z(\mathcal V),$ conditions~(1), (3) and~(4) follow.

To prove the converse implications, suppose that $UT_2\in\mathcal V$. Then $\operatorname{Id}^z(\mathcal V)
\subseteq \operatorname{Id}^z(UT_2),$ and the induced surjection of $FS_n$-modules $P_n^z(\mathcal V)\rightarrow P_n^z(UT_2)$ yields
$m_\lambda^z(\mathcal V)\ge m_\lambda^z(UT_2)$ for every $n\ge1$ and $\lambda\vdash n.$ In particular, Lemma~\ref{lem:UT2-central} gives
\[
m_{(n-1,1)}^z(\mathcal V)\ge n-1,
\] for all $n\geq 2$, so condition~(1) fails. The same lemma gives $m_{(r,r)}^z(\mathcal V)\ge1$ for every $r\ge1$.
For $\lambda=(r,r)\vdash 2r$, we have
$2r-\lambda_1-\lambda'_1=r-2,$ which is unbounded as $r$ increases. Thus condition~(3) also fails. Finally, suppose that condition~(4) holds and let
\[
g(x,y)=\sum_{i=0}^{m}\alpha_i y^i x y^{m-i}
\]
be a nonzero central polynomial of $\mathcal V$. By Lemma~\ref{lem:UT2-central},
\[
g\in\operatorname{Id}^z(\mathcal V)
\subseteq\operatorname{Id}^z(UT_2)
=\operatorname{Id}(UT_2).
\]
Evaluating at $x=e_{12}$ and $y=t e_{11}+e_{22}$, for an arbitrary $t\in F$, gives
\[
0=g(e_{12},t e_{11}+e_{22})
  =\left(\sum_{i=0}^{m}\alpha_i t^i\right)e_{12}.
\]
Since $F$ is infinite, all coefficients $\alpha_i$ must vanish, contrary to $g\ne0$. Hence condition~(4) fails as well.

Therefore, each of conditions~(1), (3) and~(4) implies~(2), completing the proof.
\end{proof}

As a consequence, we have the following result.

\begin{corollary}
Let $A$ be a PI-algebra over a field of characteristic zero. Then the multiplicities of the central cocharacters of $A$ are bounded by a constant if and only if the multiplicities of the ordinary cocharacters of $A$ are bounded by a constant.
\end{corollary}

\begin{proof}
Suppose first that $m_\lambda (A) \leq k$ for all $n\geq 1$ and every $\lambda\vdash n$. Since $m_\lambda^z(A) \leq m_\lambda(A)$, it immediately follows that the multiplicities of the central cocharacters of $A$ are also bounded by a constant.

Conversely, assume that $m_\lambda^z(A)\leq k$ for all $n\geq 1$ and every $\lambda\vdash n$. By the previous theorem, $UT_2\notin\operatorname{var}(A).$ Hence, \cite[Theorem~1]{MishchenkoRegevZaicev1999} implies that the ordinary multiplicities $m_\lambda(A)$ are bounded by a constant, completing the proof.\end{proof}

\section{Some algebras with bounded central multiplicities}

In this section, we describe the central polynomials and cocharacters of some algebras needed for the main result of this paper. An important tool is the Leibniz rule for commutators,
$$
[x_1x_2,x_3]=x_1[x_2,x_3]+[x_1,x_3]x_2.
$$

We begin with the infinite-dimensional Grassmann algebra
\[
\mathcal G
=
\langle 1,e_1,e_2,\ldots
\mid e_i e_j=-e_j e_i\rangle.
\]
We recall the following facts concerning its identities and cocharacters.

\begin{lemma}[{\cite{Brandao2010,KrakowskiRegev1973,Regev2016}}]
\label{lem:grassmann-central}
For the infinite-dimensional Grassmann algebra, we have $\operatorname{Id}(\mathcal G)=\left\langle [x_1,x_2,x_3]\right\rangle_T$ and $\operatorname{Id}^z(\mathcal G)
=
\left\langle
[x_1,x_2],\operatorname{Id}(\mathcal G)
\right\rangle^T.$ Moreover, for every $n\geq2$,
\[
\chi_n(\mathcal G)
=
\sum_{k=1}^n\chi_{(k,1^{n-k})}
\qquad\text{and}\qquad
\chi_n^z(\mathcal G)
=
\sum_{j=0}^{\lfloor (n-1)/2\rfloor}
\chi_{(n-2j,1^{2j})}.
\]
\end{lemma}

We shall also use the following well-known subalgebras of $UT_2$
\[
A_1=Fe_{11}+Fe_{12}
\qquad\text{and}\qquad
A_1^*=Fe_{22}+Fe_{12}.
\]

Their centers are zero and hence their central polynomials coincide with their polynomial identities, whose generators are given in the following lemma.

\begin{lemma}[{\cite{GLM2005}}]
\label{lem:A1_and_A1star}
For the algebras $A_1$ and $A_1^*$, we have
$\operatorname{Id}^z(A_1)
=\operatorname{Id}(A_1)
=\langle [x_1,x_2]x_3\rangle_T$  and $\operatorname{Id}^z(A_1^*)
=\operatorname{Id}(A_1^*)
=\langle x_1[x_2,x_3]\rangle_T.$ Moreover, for every $n\geq2$, $$\chi_n^z(A_1)=\chi_n^z(A_1^*)
=\chi_{(n)}+\chi_{(n-1,1)}.$$
\end{lemma}

The following result compares the central cocharacter multiplicities of a direct sum with those of its summands.

\begin{proposition}\label{prop:central-cocharacters-direct-sum}
For two algebras $A$ and $B$,  write

\[
\begin{aligned}
\chi_n^z(A)
&=\sum_{\lambda\vdash n}m_\lambda^z(A)\chi_\lambda,\quad \chi_n^z(B) =\sum_{\lambda\vdash n}m_\lambda^z(B)\chi_\lambda \quad \text{and} \quad 
\chi_n^z(A\oplus B) =\sum_{\lambda\vdash n}m_\lambda^z(A\oplus B)\chi_\lambda.
\end{aligned}
\] Then, for every $n\geq1$ and every partition $\lambda\vdash n$, $$m_\lambda^z(A),m_\lambda^z(B)
\leq
m_\lambda^z(A\oplus B)
\leq
m_\lambda^z(A)+m_\lambda^z(B).$$
\end{proposition}

\begin{proof}
Since $Z(A\oplus B)=Z(A)\oplus Z(B)$, we have $\operatorname{Id}^z(A\oplus B)= \operatorname{Id}^z(A)\cap\operatorname{Id}^z(B).$
Therefore, the natural map
$$P_n^z(A\oplus B) \rightarrow P_n^z(A)\oplus P_n^z(B)$$
is an injective homomorphism of $S_n$-modules, which provides $m_\lambda^z(A\oplus B) \leq m_\lambda^z(A)+m_\lambda^z(B).$ On the other hand, the inclusions $\operatorname{Id}^z(A\oplus B)\subseteq\operatorname{Id}^z(A) $ and $
\operatorname{Id}^z(A\oplus B)\subseteq\operatorname{Id}^z(B)$ induce surjective homomorphisms
\[
P_n^z(A\oplus B)\rightarrow P_n^z(A)
\quad\text{and}\quad
P_n^z(A\oplus B)\rightarrow P_n^z(B),
\]
which yield $m_\lambda^z(A), m_\lambda^z(B) \leq m_\lambda^z(A\oplus B)$.

\end{proof}

Making use of the previous result, we next describe the central polynomials and cocharacters of an algebra that will play a role in the next section.

\begin{lemma}\label{prop:central-polynomials-G-A1star}
Let $B=\mathcal G\oplus A_1^*$. Then
$\operatorname{Id}^z(B)
=
\left\langle
x_1[x_2,x_1],
\,
[x_1,x_2][x_3,x_4]
\right\rangle^T.$
Moreover, for every $n\geq2$,
\[
\chi_n^z(B)
=
\chi_{(n-1,1)}
+
\chi_n^z(\mathcal G).
\]
\end{lemma}

\begin{proof}
Recall that $m_{(n)}^z \leq m_{(n)}\leq d_{(n)}= 1$. Thus, Proposition \ref{prop:central-cocharacters-direct-sum} and Lemmas \ref{lem:grassmann-central} and \ref{lem:A1_and_A1star} yield
$$
\chi_n^z(B)
=
\chi_{(n-1,1)}
+
\sum_{j=0}^{\lfloor (n-1)/2\rfloor}
\chi_{(n-2j,1^{2j})}.
$$

Now let $\mathcal{C}
=
\left\langle
x_1[x_2,x_1],
\,
[x_1,x_2][x_3,x_4]
\right\rangle^T.$
Clearly, $\mathcal C\subseteq\operatorname{Id}^z(B)$ and so $$\dim_F\left(\frac{P_n}{P_n\cap \mathcal{C}}\right)\geq c_n^z(B)=(n-1)+c_n^z(\mathcal{G}).$$
For the reverse inclusion, we first observe that
$$\dim_F \frac{P_n}{P_n \cap \mathcal{C}}= \dim_F \frac{P_n \cap \operatorname{Id}^z(\mathcal{G})}{P_n \cap \mathcal{C}}+c_n ^z(\mathcal{G}).$$

The proof will be complete once we prove that the right-hand side is at most $(n-1)+c_n^z(\mathcal{G})$. We first observe that $x_4[x_1,x_2,x_3] \in \mathcal{C}$. Indeed, the linearization of $x_1[x_2,x_1]$ is $x_1[x_2,x_3]-x_3[x_1,x_2]$. Applying the endomorphism $x_1\mapsto[x_1,x_2]$, $x_2 \mapsto x_3$ and $x_3 \mapsto x_4$, we obtain $[x_1,x_2][x_3,x_4]-x_4[x_1,x_2,x_3] \in \mathcal{C}$, which implies that $x_4[x_1,x_2,x_3] \in \mathcal{C}$.

Now, we note that $[\mathcal{C}, x] \subseteq\operatorname{Id}^z(B)$, since $\mathcal{C} \subseteq\operatorname{Id}^z(B)$. In fact, we claim that $[\mathcal{C}, x] \subseteq\mathcal{C}$. To see this, it suffices to note that
$$[[x_1,x_2][x_3,x_4],x_5] =[x_1,x_2][x_3,x_4,x_5]+[x_1,x_2,x_5][x_3,x_4]\quad \mbox{ and }$$
$$[x_1[x_2,x_1],x_3]=x_1[x_2,x_1,x_3]+[x_1,x_3][x_2,x_1].$$
The claim follows since $[x_1,x_2][x_3,x_4], x_4[x_1,x_2,x_3] \in \mathcal{C}$.

Now, let $\ell(x_1,x_2,x_3)=x_1[x_2,x_3]+x_3[x_2,x_1]$ be the linearization of $x_1[x_2,x_1]$. A direct computation gives
\[
[x_1x_2,x_3]-[x_1,x_3,x_2]
=
\ell(x_1,x_2,x_3)+\ell(x_2,x_1,x_3)
\quad \mbox{
and}\quad  [x_1,x_2x_3]-[x_1,x_2,x_3]=\ell(x_2,x_1,x_3),
\]
consequently,
\begin{equation}\label{eq_consequence_linearization}
    [x_1x_2,x_3]\equiv[x_1,x_3,x_2] \quad\textnormal{and} \quad [x_1,x_2x_3]\equiv[x_1,x_2,x_3] \mod\mathcal{C}.
\end{equation}
Since $[\mathcal{C}, x] \subseteq\mathcal{C}$, these congruences may also be applied inside longer commutators.

Now, we recall that $
\operatorname{Id}^z(\mathcal{G})
=
\left\langle
[x_1,x_2],\operatorname{Id}(\mathcal{G})
\right\rangle^T$ and $\operatorname{Id}(\mathcal{G})
=\left\langle[x_1,x_2,x_3]\right\rangle_T.
$ Thus, every element of
$P_n\cap\operatorname{Id}^z(\mathcal G)$ is a linear combination of multilinear components of endomorphic
images of the polynomials $[x_1,x_2],\,
x_1[x_2,x_3,x_4],\,
[x_2,x_3,x_4]x_5$ and $x_1[x_2,x_3,x_4]x_5.$

 By \eqref{eq_consequence_linearization}, every multilinear component of an endomorphic image
of $[x_1,x_2]$ is congruent modulo $C$ to a linear combination of left-normed multilinear commutators.

For the other types, recall that $x_4[x_1,x_2,x_3] \in \mathcal{C}$, which implies that
$$x_1[x_2,x_3,x_4]x_5=x_1[x_2,x_3,x_4,x_5]+x_1x_5[x_2,x_3,x_4] \in \mathcal{C}\quad \mbox{ and }$$
$$
[x_1,x_2,x_3]x_4=[x_1,x_2,x_3,x_4]+x_4[x_1,x_2,x_3] \equiv [x_1,x_2,x_3,x_4] \mod \mathcal{C}.
$$

It follows that $\frac{P_n \cap \operatorname{Id}^z(\mathcal{G})}{P_n \cap \mathcal{C}}$ is generated by commutators of the form $[x_{i_1}, x_{i_2}, \ldots, x_{i_n}]$. The Jacobi identity yields
$$[x_{1},[x_{2},x_3]]=[x_1,x_2,x_3]-[x_1,x_3,x_2].$$
Since $[x_1,x_2,[x_3,x_4]] \in \mathcal{C}$ and $[[x_1,x_2],[x_3,x_4]]=[x_1,x_2][x_3,x_4]-[x_3,x_4][x_1,x_2]$, we obtain
$$
[[x_1,x_2],x_3,x_4] \equiv [[x_1,x_2],x_4,x_3] \mod \mathcal{C}.
$$

Recall that the space of multilinear commutators of degree $n$
is spanned by
\[
[x_i,x_n,x_{j_1},\ldots,x_{j_{n-2}}],
\qquad 1\leq i\leq n-1,
\]
where $(j_1,\ldots,j_{n-2})$ runs through all orderings of
$\{1,\ldots,n-1\}\setminus\{i\}$. Modulo $C$, the variables from the third position onward can be
permuted freely. Hence, for each fixed $i$, all these commutators are congruent modulo $C$. This implies that the quotient space $\frac{P_n \cap \operatorname{Id}^z(\mathcal{G})}{P_n \cap \mathcal{C}}$ is spanned by the $n-1$ elements
$$ [x_i, x_n, x_1, \ldots, \widehat{x_i}, \ldots, x_{n-1}], \quad 1 \leq i \leq n-1,
$$
where $\widehat{x_i}$ indicates that the variable $x_i$ is omitted. Consequently, 
$$
\dim_F \frac{P_n \cap \operatorname{Id}^z(\mathcal{G})}{P_n \cap \mathcal{C}} \leq n-1,
$$
which completes the proof.

\end{proof}

In a similar way, we can prove the following result.

\begin{lemma}\label{prop:central-polynomials-G-A1}
Let $B=\mathcal G\oplus A_1$. Then $\operatorname{Id}^z(B)
=
\left\langle
[x_1,x_2]x_1,
\,
[x_1,x_2][x_3,x_4]
\right\rangle^T.$
Moreover, for every $n\geq2$,
\[
\chi_n^z(B)
=
\chi_{(n-1,1)}+\chi_n^z(\mathcal G).
\]
\end{lemma}

Note that the algebras presented above have central multiplicities bounded by one.

\section{Characterizing varieties with central multiplicities bounded by 1}

In this section, we prove that all central cocharacter multiplicities
of $A$ are bounded by one if and only if there exist
$\alpha,\beta\in F$, not both zero, such that
\[
\alpha[x_1,x_2]x_1+\beta x_1[x_1,x_2]\in\Id^z(A).
\]
The necessity follows from the criterion
in~\eqref{eq:degree-three-central}. To prove sufficiency, we consider
the possible relations between $\alpha$ and $\beta$ and use the
algebras described in Section~4.

Throughout this section, for $k\geq2$, we write
\[
T^{(k)}=\left\langle[x_1,\ldots,x_k]\right\rangle_T,
\]
where commutators are left-normed.

\subsection{The cases $\alpha=0$ or $\beta=0$}

For $\alpha=0$, we compare $A$ with $\mathcal G\oplus A_1^*$.
The next lemma supplies the second generator of the central
$T$-space described in
Lemma~\ref{prop:central-polynomials-G-A1star}.

For the calculation below, recall that the alternation of a
polynomial $f(x_1,\ldots,x_n,y_1,\ldots,y_m)$, multilinear in
$x_1,\ldots,x_n$, is defined by
\[
\mathcal A_{x_1,\ldots,x_n}f
=
\sum_{\sigma\in S_n}\operatorname{sgn}(\sigma)
f(x_{\sigma(1)},\ldots,x_{\sigma(n)},y_1,\ldots,y_m).
\]

\begin{lemma}\label{lem:product-commutators-central}
Let $A$ be an algebra such that
$x_1[x_2,x_1]\in\operatorname{Id}^z(A)$. Then
$[x_1,x_2][x_3,x_4]\in\operatorname{Id}^z(A)$.
\end{lemma}

\begin{proof}
Linearizing $x_1[x_2,x_1]$ and taking the commutator with $x_4$,
we obtain the polynomial identity
\[
g(x_1,x_2,x_3,x_4)
=
[x_1[x_2,x_3]-x_3[x_1,x_2],x_4]
\in\operatorname{Id}(A).
\]
We use the following equality
\begin{align*}
6\bigl[[x_1,x_2][x_3,x_4],x_5\bigr]
&=
\mathcal{A}_{x_1,x_2}\mathcal{A}_{x_3,x_4}
\Bigl(
g(x_1x_2,x_3,x_4,x_5)
+3g(x_1x_2,x_3,x_5,x_4)\\
&
+g(x_1x_2,x_5,x_3,x_4)
-g(x_1x_3,x_4,x_2,x_5)
-4g(x_1x_3,x_5,x_4,x_2)\\
&
+2g(x_3x_1,x_2,x_5,x_4)
+g(x_3x_1,x_4,x_5,x_2)
-3g(x_3x_5,x_1,x_4,x_2)\\
&
-2g(x_5x_1,x_2,x_3,x_4)
+2g(x_1,x_5x_3,x_4,x_2)+g(x_2,x_5,x_4,x_1x_3)\\
&
+g(x_2,x_4,x_5,x_3x_1)
-g(x_1,x_2,x_4,x_3x_5)
\Bigr).
\end{align*}
The right-hand side is a linear combination of consequences of $g$ and therefore belongs to $\operatorname{Id}(A)$.
Since $\operatorname{char}F=0$, it follows that
$[[x_1,x_2][x_3,x_4],x_5]\in\operatorname{Id}(A)$.
Hence $[x_1,x_2][x_3,x_4]\in\operatorname{Id}^z(A)$.

\end{proof}

Together with Lemma~\ref{prop:central-polynomials-G-A1star},
this gives the required bound.

\begin{theorem}\label{thm:alpha-zero-central}
Let $A$ be an associative algebra over a field $F$ of
characteristic zero. If
$x_1[x_2,x_1]\in\operatorname{Id}^z(A)$, then
$m_\lambda^z(A)\leq1$ for every $n\geq1$ and every
partition $\lambda\vdash n$.
\end{theorem}

\begin{proof}
Let $B=\mathcal G\oplus A_1^*$.
By Lemma~\ref{lem:product-commutators-central},
$[x_1,x_2][x_3,x_4]\in\operatorname{Id}^z(A)$.
This gives
$$
\left\langle
x_1[x_2,x_1],
\,
[x_1,x_2][x_3,x_4]
\right\rangle^T
\subseteq
\operatorname{Id}^z(A).
$$
By Lemma~\ref{prop:central-polynomials-G-A1star},
$\operatorname{Id}^z(B)
=
\left\langle
x_1[x_2,x_1],
\,
[x_1,x_2][x_3,x_4]
\right\rangle^T$,
which provides
$\operatorname{Id}^z(B)\subseteq\operatorname{Id}^z(A)$.
Finally, the same lemma gives $m_\lambda^z(A)\leq m_\lambda^z(B)\leq1$ for all $n\geq1$ and all $\lambda\vdash n$.

\end{proof}

If $\beta=0$, then
$[x_1,x_2]x_1\in\operatorname{Id}^z(A)$.
The same argument applies after reversing the order of the factors in each monomial. In particular,
$[x_1,x_2][x_3,x_4]\in\operatorname{Id}^z(A)$ and
Lemma~\ref{prop:central-polynomials-G-A1} yields
$\operatorname{Id}^z(\mathcal G\oplus A_1)
\subseteq\operatorname{Id}^z(A)$.
We therefore obtain the corresponding result.

\begin{theorem}\label{thm:beta-zero-central}
Let $A$ be an associative algebra over a field $F$ of
characteristic zero. If
$[x_1,x_2]x_1\in\operatorname{Id}^z(A)$, then
$m_\lambda^z(A)\leq1$ for all $n\geq1$ and every
partition $\lambda\vdash n$.
\end{theorem}

\subsection{The case $\alpha-\beta=0$}

Here the hypothesis becomes
$[x_1^2,x_2]\in\operatorname{Id}^z(A)$.
We shall show that the multilinear identities of $\mathcal G$
are central on $A$ in degrees at least four.
We begin with the following consequence of the hypothesis.

\begin{lemma}\label{lem:block-commutators-central}
Let $A$ be an algebra such that
$[x_1^2,x_2]\in\operatorname{Id}^z(A)$. Then
$[x_1,x_2x_3x_4]\in\operatorname{Id}^z(A)$.
\end{lemma}

\begin{proof}
The hypothesis implies that
$[x_1,x_2^2]\in\operatorname{Id}^z(A)$.
Linearizing with respect to $x_2$, we obtain
\[
g(x_1,x_2,x_3)
=
[x_1,x_2x_3+x_3x_2]
\in\operatorname{Id}^z(A).
\]
Since $\operatorname{Id}^z(A)$ is a $T$-space, we have
\[
\begin{aligned}
2[x_1,x_2x_3x_4]
=g(x_1,x_3x_4,x_2)+g(x_1,x_2x_3,x_4)
-g(x_1,x_4x_2,x_3)
\in\operatorname{Id}^z(A).
\end{aligned}
\]
The result follows since $\operatorname{char}F=0$.

\end{proof}

This extends to commutators of arbitrary total degree at least four.

\begin{corollary}\label{cor:commutators-central}
Let $A$ be an algebra such that
$[x_1^2,x_2]\in\operatorname{Id}^z(A)$. If
$f,g\in F\langle X\rangle$ are multihomogeneous polynomials without
constant terms and $\deg f+\deg g\geq4$, then
$[f,g]\in\operatorname{Id}^z(A)$.
\end{corollary}

\begin{proof}
By the bilinearity of the commutator, we may assume that $f$ and $g$
are monomials. If $\deg g\geq3$, write $g=g_1g_2g_3$, where each
$g_i$ is a nonempty word. Then $[f,g]$ is an endomorphic image
of $[x_1,x_2x_3x_4]$ and hence belongs to
$\operatorname{Id}^z(A)$ by Lemma~\ref{lem:block-commutators-central}.
The case $\deg f\geq3$ follows by interchanging $f$ and $g$.

The only remaining case is $\deg f=\deg g=2$. Write $f=f_1f_2$ and
$g=g_1g_2$. Then
$$
[f_1f_2,g_1g_2]
=
[f_1,f_2g_1g_2]+[f_2,g_1g_2f_1].
$$
Both terms on the right-hand side are endomorphic images of
$[x_1,x_2x_3x_4]$. Therefore,
$[f,g]\in\operatorname{Id}^z(A)$.
\end{proof}

We also need another consequence of our assumption.

\begin{lemma}\label{lem:prefix-triple}
Let $A$ be an algebra such that
$[x_1^2,x_2]\in\operatorname{Id}^z(A)$. Then
$x_1[x_2,x_3,x_4]\in\operatorname{Id}^z(A)$.
\end{lemma}

\begin{proof}
Write $x_1\circ x_2=x_1x_2+x_2x_1$.
Linearizing $[[x_1^2,x_3],x_4]\in\operatorname{Id}(A)$
with respect to $x_1$, we obtain
\[
h(x_1,x_2,x_3,x_4)
=
[[x_1\circ x_2,x_3],x_4]
\in\operatorname{Id}(A).
\]
A direct computation gives
\begin{align*}
2[x_1[x_2,x_3,x_4],x_5]
=&h([x_1,x_4],x_5,x_2,x_3)
-h([x_1,x_4],x_5,x_3,x_2)+ h([x_2,x_3],x_1,x_4,x_5) \\
+&h([x_2,x_3],x_5,x_1,x_4)
-h([x_2,x_3],x_5,x_4,x_1)
+h(x_2\circ x_4,x_3,x_1,x_5)\\
-&h(x_3\circ x_4,x_2,x_1,x_5).
\end{align*}
Every term on the right-hand side is a consequence of
$h$ and therefore belongs to $\operatorname{Id}(A)$.
Since $\operatorname{char}F=0$, it follows that
$[x_1[x_2,x_3,x_4],x_5]\in\operatorname{Id}(A)$.
Hence $x_1[x_2,x_3,x_4]\in\operatorname{Id}^z(A)$.

\end{proof}

We now apply these two facts to the multilinear elements of
$T^{(3)}$.

\begin{lemma}\label{lem:T3-central}
Let $A$ be an algebra such that
$[x_1^2,x_2]\in\operatorname{Id}^z(A)$. Then, for every $n\geq4$, we have
\[
P_n\cap T^{(3)}
\subseteq
P_n\cap\operatorname{Id}^z(A).
\]
\end{lemma}

\begin{proof}
The space $P_n\cap T^{(3)}$ is spanned by polynomials of the form $p[f,g,h]q$, where $f,g,h$ are nonempty multilinear words and $p,q$ are possibly absent. Set $c=[f,g,h]$. 

Suppose first that $p$ is present. By Lemma~\ref{lem:prefix-triple}, $pc\in\operatorname{Id}^z(A)$. If $q$ is absent, there is nothing to prove. Otherwise, $pcq=[pc,q]+qpc$. The first summand is central by Corollary~\ref{cor:commutators-central}, since its total degree is $n\geq4$, while the second is central by Lemma~\ref{lem:prefix-triple}, applied to $(qp)c$. 

Now suppose that $p$ is absent and $q$ is present. In this case, $cq=[c,q]+qc$. The first summand belongs to $\operatorname{Id}^z(A)$ by  Corollary~\ref{cor:commutators-central} and the second belongs to $\operatorname{Id}^z(A)$ by Lemma~\ref{lem:prefix-triple}.

Finally, if both $p$ and $q$ are absent, then $c=[[f,g],h]$. Since its total degree is $n\geq4$, Corollary~\ref{cor:commutators-central} gives $c\in\operatorname{Id}^z(A)$. Therefore, every spanning polynomial of $P_n\cap T^{(3)}$ is central, which proves the desired inclusion.

\end{proof}

Since $T^{(3)}=\Id(\mathcal G)$, this inclusion allows us to compare the central cocharacters of $A$ with the ordinary cocharacters of $\mathcal G$. Together with the criterion in~\eqref{eq:degree-three-central}, it yields the following result.

\begin{theorem}\label{thm:square-central}
Let $A$ be an associative algebra over a field $F$ of characteristic zero. If
$[x_1^2,x_2]\in\operatorname{Id}^z(A)$, then
$m_\lambda^z(A)\leq1$ for all $n\geq1$ and every
partition $\lambda\vdash n$.
\end{theorem}

\begin{proof}
For $n\in\{1,2\}$, the conclusion follows since $P_1$ and $P_2$
have multiplicities bounded by one. In degree three, it remains
only to consider the partition $\lambda=(2,1)$. Since
\[
[x_1^2,x_2]=[x_1,x_2]x_1+x_1[x_1,x_2]
\in\operatorname{Id}^z(A),
\]
the criterion in~\eqref{eq:degree-three-central},
with $\alpha=\beta=1$, gives $m_{(2,1)}^z(A)\leq1$.

Now assume that $n\geq4$. By Lemma~\ref{lem:T3-central}
and the equality $\Id(\mathcal G)=T^{(3)}$, we have
$P_n\cap\Id(\mathcal G)\subseteq P_n\cap\Id^z(A)$.
Consequently, $m_\lambda^z(A)\leq m_\lambda(\mathcal G)\leq1$, for every $\lambda\vdash n.$
\end{proof}

\subsection{The case $\alpha+\beta=0$}

We now consider the case $[x_1,x_2,x_1]\in\Id^z(A)$
and establish the following consequence.

\begin{lemma}\label{conseqcomutadortriplo}
Let $A$ be an algebra such that $[x_1,x_2,x_1]\in\Id^z(A)$. Then
$[x_1,x_2,x_3]\in\Id^z(A)$.
\end{lemma}

\begin{proof}
Observe that the linearization of the polynomial $[x_1,x_2,x_1]$ gives
$[x_1,x_2,x_3]+[x_3,x_2,x_1]$. Consider
$
I=\left\langle [x_1,x_2,x_1]\right\rangle^T.
$
It follows that
$$
[x_1,x_2,x_3]
\equiv-[x_3,x_2,x_1]
=[x_2,x_3,x_1]
\pmod I.
$$
Similarly, applying the linearization after permuting the variables, we obtain
$$
[x_1,x_2,x_3]
=-[x_2,x_1,x_3]
\equiv [x_3,x_1,x_2]
\pmod I.
$$
Therefore, the Jacobi identity gives
$$
0=[x_1,x_2,x_3]+[x_2,x_3,x_1]+[x_3,x_1,x_2]
\equiv 3[x_1,x_2,x_3]
\pmod I.
$$
Since $\operatorname{char}F=0$, we conclude that
$[x_1,x_2,x_3]\in I$.

\end{proof}

We shall also need the following inclusion involving commutator ideals.

\begin{lemma}\label{lem:T2T3}
$T^{(2)}T^{(3)}+T^{(3)}T^{(2)}
\subseteq T^{(4)}.$
\end{lemma}

\begin{proof}
Let $f(x_1,x_2,x_3,x_4,x_5) = [x_1,x_2,x_3][x_4,x_5]$. Modulo $T^{(4)}$, it follows that
\begin{equation}\label{relation1}
0 \equiv [x_1,x_2,x_3x_4,x_5]
\equiv [x_1,x_2,x_3][x_4,x_5]
+[x_3,x_5][x_1,x_2,x_4].
\end{equation}
Since $[x_1,x_2,x_3]$ is central modulo $T^{(4)}$, we have $f(x_1,x_2,x_3,x_4,x_5) \equiv - f(x_1,x_2,x_4,x_3,x_5)$.

Now, expanding $[x_1,x_2x_4,x_3,x_5]$ via the Leibniz rule yields
\begin{align*}
0\equiv [x_1,x_2][x_4,x_3,x_5]
+[x_1,x_2,x_5][x_4,x_3]
+[x_1,x_2,x_3][x_4,x_5]\\
+[x_2,x_5][x_1,x_4,x_3]
+[x_2,x_3,x_5][x_1,x_4]
+[x_2,x_3][x_1,x_4,x_5].
\end{align*}
By \eqref{relation1}, we obtain
$[x_1,x_2][x_4,x_3,x_5]
\equiv -[x_2,x_3,x_5][x_1,x_4]
\pmod{T^{(4)}}$.
Combined with the centrality of $[x_1,x_2,x_3]$
modulo $T^{(4)}$, this leads to
$f(x_1,x_2,x_3,x_4,x_5)
\equiv -f(x_5,x_2,x_3,x_4,x_1)$.
Notice that the relations obtained above imply that $f$
is alternating modulo $T^{(4)}$.
On the other hand, the Jacobi identity yields 
$$
f(x_1,x_2,x_3,x_4,x_5)
+f(x_2,x_3,x_1,x_4,x_5)
+f(x_3,x_1,x_2,x_4,x_5)=0.
$$
Therefore, since $f$ is alternating modulo $T^{(4)}$, all three summands are equivalent, yielding $3f(x_1,x_2,x_3,x_4,x_5) \equiv 0 \pmod{T^{(4)}}$. Hence, $[x_1,x_2,x_3][x_4,x_5] \in T^{(4)}$.

Now, the identity $[x_6x_4,x_5]=x_6[x_4,x_5]+[x_6,x_5]x_4$ provides
$$
[x_1,x_2,x_3]x_6[x_4,x_5]
=
[x_1,x_2,x_3][x_6x_4,x_5]
-[x_1,x_2,x_3][x_6,x_5]x_4
\in T^{(4)}.
$$
Hence, $T^{(3)}T^{(2)}\subseteq T^{(4)}$. Finally, observe that
$$
[x_1,x_2]x_6[x_3,x_4,x_5]
\equiv
[x_3,x_4,x_5][x_1,x_2]x_6
\equiv0\pmod{T^{(4)}}.
$$
It follows that $T^{(2)}T^{(3)}\subseteq T^{(4)}$ and consequently $T^{(2)}T^{(3)}+T^{(3)}T^{(2)}\subseteq T^{(4)}$.

\end{proof}

We can now show that every element of $T^{(3)}$ is central on $A$.

\begin{lemma}\label{triple-ideal-central}
Let $A$ be an algebra such that $[x_1,x_2,x_1]\in\operatorname{Id}^z(A)$. Then
$T^{(3)}\subseteq\operatorname{Id}^z(A).$

\end{lemma}

\begin{proof}
By Lemma~\ref{conseqcomutadortriplo},
$[x_1,x_2,x_3]$ belongs to $\Id^z(A)$. Hence
$[x_1,x_2,x_3,x_4]\in\Id(A)$ and therefore $T^{(4)}\subseteq\Id(A)$.

Let $c=[x_1,x_2,x_3]$. Since $\Id^z(A)$ is invariant under endomorphisms, it is enough to consider the polynomials
$c$, $x_4c$, $cx_5$ and $x_4cx_5$. The first one is central by
Lemma~\ref{conseqcomutadortriplo}. For the remaining cases, we have
$$
\begin{aligned}
[x_4c,x_6]
    &=x_4[c,x_6]+[x_4,x_6]c,\\
[cx_5,x_6]
    &=c[x_5,x_6]+[c,x_6]x_5,\\
[x_4cx_5,x_6]
    &=x_4c[x_5,x_6]+x_4[c,x_6]x_5+[x_4,x_6]cx_5.
\end{aligned}
$$
Since $[c,x_6]\in T^{(4)}$, Lemma~\ref{lem:T2T3} shows that every term on the right-hand side of these equalities belongs to $T^{(4)}\subseteq\Id(A)$. Consequently,
$x_4c$, $cx_5$ and $x_4cx_5$ belong to $\Id^z(A)$ and hence $T^{(3)}\subseteq\Id^z(A)$. \end{proof}

Comparing the central cocharacters of $A$ with the ordinary cocharacters of $\mathcal G$, we obtain the following result.

\begin{theorem}\label{central-triple-multiplicities}
Let $A$ be an associative algebra over a field $F$ of
characteristic zero. If
$[x_1,x_2,x_1]\in\operatorname{Id}^z(A)$, then
$m_\lambda^z(A)\leq1$ for every $n\geq1$ and every
partition $\lambda\vdash n$.
\end{theorem}

\begin{proof}
Lemma \ref{triple-ideal-central} implies
$\Id(\mathcal G)=T^{(3)}\subseteq\Id^z(A)$.
This inclusion implies that
$P_n\cap\Id(\mathcal G)\subseteq P_n\cap\Id^z(A)$.
Consequently, $m_\lambda^z(A)\leq m_\lambda(\mathcal G)$,
for every $\lambda\vdash n$.
Finally, since Lemma \ref{lem:grassmann-central} establishes that
$m_\lambda(\mathcal G)\leq1$, we conclude that
$m_\lambda^z(A)\leq1$. \end{proof}

\subsection{The case $\alpha\beta(\alpha+\beta)(\alpha-\beta)\neq0$}

We now assume that
\[
\alpha[x_1,x_2]x_1+\beta x_1[x_1,x_2]\in\Id^z(A)
\quad \text{for} \quad
\alpha\beta(\alpha+\beta)(\alpha-\beta)\neq0.
\]
We shall compare the central cocharacters of $A$ with those of
$\mathcal G$ in degrees at least four.

Set $t=\beta/\alpha$, so that $t(t-1)(t+1)\neq0$.
Interchanging $x_1$ and $x_2$, multiplying by $-1/\alpha$ and
linearizing with respect to $x_2$, we obtain
\[
q_t(x_1,x_2,x_3)
=
[x_1,x_2]x_3+[x_1,x_3]x_2
+t\bigl(x_2[x_1,x_3]+x_3[x_1,x_2]\bigr)
\in\Id^z(A).
\]
We therefore set
\[
\mathcal I_t
=
\bigl\langle[q_t(x_1,x_2,x_3),x_4]\bigr\rangle_T
\subseteq\Id(A).
\]
The next two lemmas provide the central polynomials needed for
the comparison with $\mathcal G$.

\begin{lemma}\label{lem:generic-long-commutator} For every $t\notin \{0,1,-1\}$ we have $[[x_1,x_2x_3x_4],x_5]\in\mathcal I_t$.
\end{lemma}

\begin{proof}
    Consider the multilinear polynomials $c = [x_1,x_2x_3x_4]$ and $\tilde{q_t}=q_t(x_1,x_2x_3,x_4)$. We shall prove that $FS_4\tilde{q_t} = FS_4c$. First, we show that $FS_4\tilde{q_t} \subseteq FS_4c$.
    
    Recall that the space $P_n$ is isomorphic to $FS_n$ as an $S_n$-module via the linear isomorphism
    $\sigma \mapsto x_{\sigma(1)}x_{\sigma(2)}\cdots x_{\sigma(n)}.$
    Considering $\gamma = (1\,2\,3\,4)$, it is clear that $FS_4c \cong FS_4(1-\gamma)$. 
    We now determine the dimension and the irreducible decomposition of $FS_4c$. Consider the $FS_4$-homomorphism
    $$ \varphi:FS_4\rightarrow FS_4, \qquad a\mapsto a(1-\gamma). $$ 
    Its image is $FS_4c$, whereas $\ker \varphi= \{a\in FS_4\mid a\gamma=a\}.$ Let $C_4=\langle\gamma\rangle$. Since $[S_4:C_4]=6,$ choose $g_1,\ldots,g_6\in S_4$ such that $S_4=\bigcup_{i=1}^6 g_iC_4.$  We claim that $\mathcal B=\left\{ g_i(1+\gamma+\gamma^2+\gamma^3) \;\middle|\; 1\leq i\leq 6 \right\}$ is a basis of $\ker \varphi$. It is clear that every element of $\mathcal B$ belongs to $\ker \varphi$. Now, let $a \in \ker \varphi$. Since $a\gamma=a$, comparing the coefficients gives $\alpha_\sigma=\alpha_{\sigma\gamma}$ for every $\sigma\in S_4$. Therefore, the coefficients of $a$ are constant on each coset $g_iC_4$, for $1\leq i\leq 6$. Hence,
\begin{align*}
a =\sum_{\sigma\in S_4}\alpha_\sigma\sigma = \sum_{i=1}^6\sum_{j=0}^3
\alpha_{g_i\gamma^j}g_i\gamma^j= 
\sum_{i=1}^6
\alpha_{g_i}g_i(1+\gamma+\gamma^2+\gamma^3).
\end{align*}
This proves that $\mathcal B$ spans $\ker \varphi$. Finally, since the cosets $g_iC_4$ are pairwise disjoint, the elements of $\mathcal B$ are linearly independent. Therefore, $\mathcal B$ is a basis of $\ker \varphi$ and consequently $\dim_F\ker \varphi=6.$

Let $F_{\mathrm{triv}}$ denote the trivial $FC_4$-module and recall that $F_{\mathrm{triv}} \uparrow S_4 =FS_4\otimes_{FC_4}F_{\mathrm{triv}}$. The basis $\mathcal B$ yields an $FS_4$-module isomorphism
$$
FS_4\otimes_{FC_4}F_{\mathrm{triv}}
\rightarrow\ker \varphi,
\qquad
g\otimes1\mapsto g(1+\gamma+\gamma^2+\gamma^3).
$$
Therefore, $ \ker \varphi \cong F_{\mathrm{triv}} \uparrow S_4$ as an $FS_4$-module and hence $\chi_{\ker \varphi}=\chi_{\text{triv}}\!\uparrow{S_4}.$ By Frobenius reciprocity, the multiplicity of $S_\lambda$ in
$\ker \varphi$ is
\begin{align*}
\dim_F (S_\lambda)^{C_4} = 
\frac{1}{4}
\left(
\chi_\lambda(id)
+2\chi_\lambda((1\,2\,3\,4))
+\chi_\lambda((1\,3)(2\,4))
\right).
\end{align*}
The character table of $S_4$ shows that $\ker \varphi \cong S_{(4)} \oplus S_{(2,2)} \oplus S_{(2,1,1)}. $ On the other hand, the regular $FS_4$-module decomposes as
\begin{equation}\label{equation_FS4}
FS_4\cong S_{(4)} \oplus 3S_{(3,1)} \oplus 2S_{(2,2)} \oplus 3S_{(2,1,1)} \oplus S_{(1^4)}.    
\end{equation}
Moreover, since $FS_4c=\operatorname{Im}\varphi \cong FS_4/\ker \varphi$, we conclude that
$FS_4c \cong 3S_{(3,1)} \oplus S_{(2,2)} \oplus 2S_{(2,1,1)} \oplus S_{(1^4)}.$

Now, for the inclusion $FS_4\tilde{q_t}\subseteq FS_4c$, it is sufficient to note that
\begin{align*}
\tilde{q_t}
={}&[x_1,x_2x_3x_4]
 +(1-t)[x_2,x_3x_4x_1]
 +(1-t)[x_3,x_4x_1x_2]
+[x_1,x_4x_2x_3]
 +(1-t)[x_4,x_2x_3x_1].
\end{align*}

Let $m_\lambda$ denote the multiplicity of $S_\lambda$ in $FS_4\widetilde q_t$. Since
$FS_4\widetilde q_t\subseteq FS_4c$, complete reducibility
and the above decomposition of $FS_4c$ give
\[
m_{(3,1)}\leq3,\qquad
m_{(2,2)}\leq1,\qquad
m_{(2,1,1)}\leq2,\qquad
m_{(1^4)}\leq1.
\]

We show that equality holds in all these inequalities. For each partition $\lambda\vdash4$, fix a standard tableau $T_\lambda$ and let
$$ e_{T_\lambda} = \left(\sum_{\rho\in R_{T_\lambda}}\rho\right) \left( \sum_{\tau\in C_{T_\lambda}} \operatorname{sgn}(\tau)\tau \right)$$
be the corresponding essential idempotent. Recall that, for every $FS_4$-module $M$,
$m_\lambda(M)=\dim_F e_{T_\lambda}M.$ We use this fact to compute the multiplicities of the irreducible constituents of $FS_4\tilde{q_t}$. First, consider $\lambda=(3,1)$ and take
$$
T_\lambda=
\begin{array}{|c|c|c|}
\hline
1&2&3\\ \hline
4&\multicolumn{2}{c}{}\\ \cline{1-1}
\end{array}.
$$
It follows that
$ e_{T_\lambda} = \left(\sum_{\rho\in S_{3}}\rho\right)(1-(14)).$ Set
$$
v_1=e_{T_\lambda}\tilde{q_t},\qquad
v_2=e_{T_\lambda}(12)\tilde{q_t},\qquad
v_3=e_{T_\lambda}(13)\tilde{q_t},
$$
and suppose that $\lambda_1v_1+\lambda_2v_2+\lambda_3v_3=0.$
Comparing the coefficients of $x_1x_2x_3x_4$,
$x_4x_1x_2x_3$ and $x_1x_2x_4x_3$ respectively, we obtain
$$3t\lambda_1+t\lambda_2=0, \quad
-3\lambda_1-\lambda_3=0 \quad \text{and} \quad (1-2t)\lambda_1-2t\lambda_2-\lambda_3=0.
$$
Since $t\neq0$, the first two equations give
$\lambda_2=\lambda_3=-3\lambda_1.$ Substituting these values into the third equation yields $4(t+1)\lambda_1=0.$ Since $t\neq-1$, we obtain $\lambda_1=\lambda_2=\lambda_3=0.$ Thus, $v_1,v_2,v_3$ are linearly independent and hence $m_{(3,1)}=3.$

Next, consider $\lambda=(2,2)$ and take
$$
T_\lambda=
\begin{array}{|c|c|}
\hline
1&2\\ \hline
3&4\\ \hline
\end{array}
$$
In this case, $e_{T_\lambda} = (1+(12))(1+(34))(1-(13))(1-(24)).$ Since the coefficient of $x_1x_2x_3x_4$ in $e_{T_\lambda}(12)\tilde{q_t}$ is $t+1$ and $t\neq-1$, we obtain $e_{T_\lambda}(12)\tilde{q_t}\neq0.$ Therefore, $m_{(2,2)}=1.$

Now, consider $\lambda=(2,1,1)$ and take
$$
T_\lambda=
\begin{array}{|c|c|}
\hline
1&2\\ \hline
3&\multicolumn{1}{c}{}\\ \cline{1-1}
4&\multicolumn{1}{c}{}\\ \cline{1-1}
\end{array}
$$
Let $v_1=e_{T_\lambda}(12)\tilde{q_t}$ and $v_2=e_{T_\lambda}(24)\tilde{q_t}$ and suppose that $\lambda_1v_1+\lambda_2v_2=0.$ Comparing the coefficients of $x_1x_2x_3x_4$ and $x_1x_3x_2x_4$ respectively, we have
$$
(t+1)\lambda_1=0 \quad\text{and}\quad (t-3)\lambda_1+(t-1)\lambda_2=0.
$$
Since $t\neq-1$ and $t\neq 1$, we obtain $\lambda_1=\lambda_2=0$. Thus, $v_1,v_2$ are
linearly independent, which yields
$m_{(2,1,1)}=2.$

Finally, consider $\lambda=(1^4)$. In this case,
$e_{T_\lambda}\tilde{q_t}=
2(t+1)\operatorname{St}_4(x_1,x_2,x_3,x_4),$ and since $t\neq-1$, this polynomial is nonzero. Therefore, $m_{(1^4)}=1.$

Since $FS_4\tilde{q_t}\cong 3S_{(3,1)} \oplus S_{(2,2)} \oplus 2S_{(2,1,1)} \oplus S_{(1^4)}$ and
$FS_4\tilde{q_t}\subseteq FS_4c$, we conclude that $FS_4\tilde{q_t}=FS_4c.$ In particular, there exist scalars $\lambda_\sigma\in F$, for $\sigma\in S_4$, such that
$c=\sum_{\sigma\in S_4}\lambda_\sigma\,\sigma \tilde{q_t}.$ It follows that
$$[c,x_5] = \sum_{\sigma \in S_4}\lambda_\sigma[\sigma \tilde{q_t},x_5].$$
Now, for each $\sigma\in S_4$, consider the endomorphism defined by
$x_1\mapsto x_{\sigma(1)}$,
$x_2\mapsto x_{\sigma(2)}x_{\sigma(3)}$,
$x_3\mapsto x_{\sigma(4)}$ and $
x_4\mapsto x_5,$ which gives
$[q_t(x_1,x_2,x_3),x_4]\mapsto[\sigma\tilde{q_t},x_5]$.
Recall that $\mathcal{I}_t$ is the $T$-ideal generated by
$[q_t(x_1,x_2,x_3),x_4]$. It follows that
$[\sigma\tilde{q_t},x_5]\in\mathcal{I}_t$ for every $\sigma\in S_4$.
Therefore, $[[x_1,x_2x_3x_4],x_5]=[c,x_5]\in\mathcal{I}_t,$ which completes the proof. \end{proof}

The preceding identity also yields the centrality of
$[x_1,x_2,x_3]x_4$ modulo $\mathcal I_t$.

\begin{lemma}\label{lem:generic-triple-context}
For every $t\notin \{0,1,-1\}$ we have  $[[x_1,x_2,x_3]x_4,x_5]\in\mathcal I_t$.
\end{lemma}

\begin{proof}
Note that
$$
\begin{aligned}
[[x_1,x_2,x_3]x_4,x_5] &= [[x_1,x_3x_4x_5],x_2] - [[x_1,x_4x_3x_5],x_2] - [[x_2,x_3x_4x_5],x_1] + [[x_2,x_4x_3x_5],x_1] \\
& - [[x_3,x_1x_2x_4],x_5] + [[x_3,x_2x_1x_4],x_5] + [[x_3,x_5x_1x_2],x_4] - [[x_3,x_5x_2x_1],x_4]\\
& - [[x_4,x_5x_1x_2],x_3] + [[x_4,x_5x_2x_1],x_3].
\end{aligned}
$$

Therefore, by Lemma~\ref{lem:generic-long-commutator}, $[[x_1,x_2,x_3]x_4,x_5]\in \mathcal{I}_{t}$. \end{proof}

Let $$I
=
\left\langle
[x_1,x_2x_3x_4],
\,
[x_1,x_2,x_3]x_4
\right\rangle^T.$$ Since $\mathcal I_t\subseteq\Id(A)$,
Lemmas~\ref{lem:generic-long-commutator}
and~\ref{lem:generic-triple-context} give
$I\subseteq\Id^z(A)$.
Using the description
$\Id^z(\mathcal G)=T^{(3)}+\langle[x_1,x_2]\rangle^T$
from Lemma~\ref{lem:grassmann-central}, we obtain the following
inclusion.

\begin{proposition}\label{prop:generic-Grassmann-inclusion}
For every $n\geq4$, we have
$P_n\cap\operatorname{Id}^z(\mathcal G)
\subseteq P_n\cap\operatorname{Id}^z(A).$

\end{proposition}

\begin{proof}
By the previous discussion, it is sufficient to show that $P_n \cap \operatorname{Id}^z(\mathcal{G}) \subseteq P_n \cap I$.

First, note that $P_n \cap \langle [x_1, x_2] \rangle^T$ is spanned by multilinear commutators $[f, g]$, where $f$ and $g$ are nonempty monomials with $\deg f + \deg g = n \ge 4$. Using the identity $[x_1 x_2, x_3] = [x_1, x_2 x_3] + [x_2, x_3 x_1]$ repeatedly, any such commutator can be expressed as a linear combination of terms of the form $[x_i, h]$, where $h$ is a monomial of degree $n-1$. Since $n\geq4$, each such monomial $h$ has degree at least
three. Hence every term $[x_i,h]$ is a consequence of $[x_1,x_2x_3x_4]$. Therefore, $[f,g]\in I$. It follows that
$$
P_n \cap \langle [x_1, x_2] \rangle^T \subseteq P_n \cap  I.
$$

Next, we consider $P_n \cap T^{(3)}$. This space is spanned by polynomials of the form $p[f, g, h]q$, where $f, g, h$ are nonempty multilinear monomials and $p, q$ are possibly absent. 
If both $p$ and $q$ are absent, then
$[f,g,h]=[fg,h]-[gf,h]\in I$ by the first part of the proof. Moreover, if $p$ is empty and $q$ is nonempty, then $[f,g,h]q$ is an evaluation of $[x_1, x_2, x_3]x_4 \in  I$. 
Finally, if $p$ is nonempty, note that 
$$
p[f,g,h]q = [f,g,h]qp + [p, [f,g,h]q].
$$
Since $[p, [f,g,h]q]$ expands into a linear combination of commutators of degree $n \ge 4$, both terms belong to $ I$. Thus, $p[f,g,h]q\in I$. Therefore, for $n \geq 4$, we conclude $P_n \cap \operatorname{Id}^z(\mathcal{G}) \subseteq P_n \cap  I.$ \end{proof}

Combining this inclusion with the criterion in \eqref{eq:degree-three-central}, we obtain the desired bound on the central multiplicities.

\begin{theorem}\label{thm:generic-multiplicity-one}
Let $A$ be an associative algebra over a field $F$ of
characteristic zero. If $\alpha[x_1,x_2]x_1+\beta x_1[x_1,x_2]\in\operatorname{Id}^z(A)$ for some $\alpha,\beta\in F$ satisfying $\alpha\beta(\alpha+\beta)(\alpha-\beta)\neq0$,
then $m_\lambda^z(A)\leq1$ for every $n\geq1$
and every partition $\lambda\vdash n$.
\end{theorem}

\begin{proof}
For $n\in\{1,2,3\}$, the conclusion follows as in the previous cases. Moreover, for $n\geq 4$, by Proposition~\ref{prop:generic-Grassmann-inclusion}, we obtain $P_n\cap\operatorname{Id}^z(\mathcal G)
\subseteq P_n\cap\operatorname{Id}^z(A)$. Consequently, $m_\lambda^z(A)\leq m_\lambda^z(\mathcal G)\leq1$ for every $\lambda\vdash n$. \end{proof}

\subsection{The main result}

We now combine the preceding cases with the criterion in \eqref{eq:degree-three-central} to obtain a central analogue of the Ananin-Kemer theorem. The characterization below shows that the multiplicity-one condition for the entire sequence of central
cocharacters is determined in degree three: the single inequality $m_{(2,1)}^z(A)\leq1$ already forces all central cocharacter multiplicities to be bounded by one.

\begin{theorem}\label{thm:main-central-multiplicity-one}
Let $A$ be an associative algebra over a field $F$ of characteristic zero and let
\[
\chi_n^z(A)=\sum_{\lambda\vdash n}m_\lambda^z(A)\chi_\lambda
\]
be its $n$th central cocharacter. Then $m_\lambda^z(A)\leq1$
for every $n\geq1$ and every partition $\lambda\vdash n$ if and
only if there exist $\alpha,\beta\in F$, not both zero, such that $\alpha[x_1,x_2]x_1+\beta x_1[x_1,x_2]
\in\operatorname{Id}^z(A).$
\end{theorem}

\begin{proof}
Suppose first that $m_\lambda^z(A)\leq1$ for every $n\geq1$
and every partition $\lambda\vdash n$. In particular,
$m_{(2,1)}^z(A)\leq1$ and the existence of the required
central polynomial follows from the criterion
in~\eqref{eq:degree-three-central}. Conversely, suppose that
$\alpha[x_1,x_2]x_1+\beta x_1[x_1,x_2]
\in\operatorname{Id}^z(A)$,
where $(\alpha,\beta)\neq(0,0)$.

If $\alpha=0$, then $\beta\neq0$ and
$x_1[x_1,x_2]\in\operatorname{Id}^z(A)$.
The conclusion follows from Theorem~\ref{thm:alpha-zero-central}.
If $\beta=0$, then $\alpha\neq0$ and
$[x_1,x_2]x_1\in\operatorname{Id}^z(A)$.
In this case, Theorem~\ref{thm:beta-zero-central} gives the result.

Assume now that $\alpha\beta\neq0$. If $\alpha=\beta$, then
$[x_1^2,x_2]\in\operatorname{Id}^z(A)$ and the conclusion
follows from Theorem~\ref{thm:square-central}.
If $\alpha=-\beta$, then
$[x_1,x_2,x_1]\in\operatorname{Id}^z(A)$,
and Theorem~\ref{central-triple-multiplicities} gives the result.

The remaining case
$\alpha\beta(\alpha+\beta)(\alpha-\beta)\neq0$
is covered by Theorem~\ref{thm:generic-multiplicity-one}.\end{proof}

The following example shows that all central cocharacter multiplicities of an algebra may be bounded by one even when some ordinary cocharacter multiplicity is greater than one.

\begin{example}
 According to 
Lemma~\ref{prop:central-polynomials-G-A1star}, all central cocharacter multiplicities of $\mathcal G\oplus A_1^*$ are bounded by one. On the other hand, it is straightforward to see that $m_{(2,1)}(\mathcal G\oplus A_1^*)=2$.
\end{example}

\vspace{0.2cm}
 \textbf{Declaration of generative AI and AI-assisted technologies in the manuscript preparation process}
\vspace{0.2cm}

During the preparation of this work, the authors used ChatGPT (OpenAI) in order to help identify relevant references, improve the language and readability of the manuscript, and assist in checking proofs and algebraic computations. After using this tool/service, the authors reviewed and edited the content as needed and take full responsibility for the content of the published article.

\end{document}